\documentclass[12pt]{amsart}

\usepackage{amssymb}
\usepackage{float}
\usepackage{graphicx}
\usepackage[cmtip,all]{xy}

\newtheorem{theorem}{Theorem}[section]
\newtheorem{prop}[theorem]{Proposition}
\newtheorem{lemma}[theorem]{Lemma}

\theoremstyle{definition}

\theoremstyle{remark}
\newtheorem{remark}[theorem]{Remark}

\numberwithin{equation}{section}

\newcommand{\abs}[1]{\left\lvert#1\right\rvert}

\newcommand{\ord}{\textup{ord}}

\newcommand{\CC}{\mathcal{C}}
\newcommand{\CH}{\mathcal{H}}
\newcommand{\CX}{\mathcal{X}}
\newcommand{\CY}{\mathcal{Y}}

\begin{document}

\title[On maximal curves]
{On maximal curves which are not Galois subcovers of the Hermitian curve}


\author{Iwan Duursma}
\address{Department of Mathematics \\
University of Illinois at Urbana-Champaign \\
273 Altgeld Hall, MC-382 \\
1409 W. Green Street \\
Urbana, Illinois 61801, USA}
\email{duursma@math.uiuc.edu}

\author{Kit-Ho Mak}
\address{Department of Mathematics \\
University of Illinois at Urbana-Champaign \\
273 Altgeld Hall, MC-382 \\
1409 W. Green Street \\
Urbana, Illinois 61801, USA}
\email{mak4@illinois.edu}


\begin{abstract}
We show that the generalized Giulietti-Korchm{\'a}ros curve 
defined over $\mathbb{F}_{q^{2n}}$, for $n\geq 3$ odd and $q\geq 3$, is not a Galois subcover of the Hermitian curve over $\mathbb{F}_{q^{2n}}$. This answers a question raised by Garcia, G{\"u}neri and Stichtenoth. 
\end{abstract}

\maketitle

\section{Introduction and Statements of Results}

Let $\mathbb{F}_{q^2}$ be the finite field with $q^2$ elements, and let $\CX$ be a projective, nonsingular, geometrically irreducible curve (hereafter referred to as a \textit{curve}) defined over $\mathbb{F}_{q^2}$. We say that $\CX$ is $\mathbb{F}_{q^2}$-maximal if the number of its rational points attains the Hasse-Weil upper bound
\begin{equation*}
\abs{\CX(\mathbb{F}_{q^2})}=q^2+1+2g(\CX)q,
\end{equation*}
where $g(\CX)$ is the genus of $\CX$. The most important example of a maximal curve is the Hermitian curve $\CH$,  which is defined over $\mathbb{F}_{q^2}$ by the equation
\begin{equation*}
y^{q}+y=x^{q+1}.
\end{equation*}
It has genus $\frac{1}{2}q(q-1)$. By the work of \cite{FuTo96,Iha81,StXi95}, the genus of any maximal curve $\CX$ satisfies
\begin{equation*}
g(\CX) \in [0,(q-1)^2/4]\cup\{q(q-1)/2\}.
\end{equation*}
Therefore the Hermitian curve has the largest possible genus that a maximal curve can have. It is shown in \cite{RuSt94} that the Hermitian curve is the unique maximal curve having genus $\frac{1}{2}q(q-1)$. More information about maximal curves can be found in \cite{AbGa04,AbTo99,FGT97,Gar01,GiKo09,KoTo01,KoTo02} and their references.

A curve $\CC$ is called a \textit{subcover} of $\CX$ over $\mathbb{F}_{q^2}$ (or equivalently $\CX$ is a \textit{cover} of $\CC$) if there exists a surjective map $\phi:\CX\longrightarrow\CC$ with $\CC$, $\CX$ and $\phi$ defined over $\mathbb{F}_{q^2}$. By a result known as Serre's theorem (see \cite[Proposition 6]{Lac87}), any subcover of a $\mathbb{F}_{q^2}$-maximal curve is $\mathbb{F}_{q^2}$-maximal. Most of the known maximal curves are subcovers of the Hermitian curve $\CH$, and systematic studies on subcovers of $\CH$ can be found in \cite{CKT99,CKT00,GSX00}.

The first example of a maximal curve which is not a Galois subcover of the Hermitian curve is the curve $y^9-y = z^7$ over $\mathbb{F}_{3^6}$ discovered by Garcia and Stichtenoth \cite{GaSt06}. It is the special case with $q=3$, $n=3$ of the curve $\CX_n$ with defining equation
\begin{equation}\label{Xneqn}
y^{q^2}-y=z^{\frac{q^n+1}{q+1}}
\end{equation}
over $\mathbb{F}_{q^{2n}}$, for $q\geq 2$ and odd $n\geq 3$. The curve $\CX_n$ was shown to be maximal in \cite{ABQ09}. As another example, an unpublished calculation by Rains and Zieve states that the Ree curve of genus $g=15$ over $\mathbb{F}_{3^6}$ is not a Galois subcover of the Hermitian curve over the same field.

In \cite{GiKo09}, Giulietti and Korchm{\'a}ros give an example of a maximal curve, now called the GK curve, that is not covered by the Hermitian curve. The GK curve has been generalized by Garcia, G{\"u}neri and Stichtenoth in \cite{GGS10}. The generalized GK curves $\CC_n$ are maximal curves over $\mathbb{F}_{q^{2n}}$ for a prime power $q$ and odd $n\geq 3$ (\cite{GGS10}, see also \cite{Duu10}). They have defining equations
\begin{equation}\label{GGKeqn}
\begin{aligned}
x^q+x &= y^{q+1} \\
y^{q^2}-y &= z^{\frac{q^n+1}{q+1}}.
\end{aligned}
\end{equation}
It is shown in \cite{GGS10} that the curves with $n=3$ are isomorphic to those given originally by Giulietti and Korchm{\'a}ros. It is not known whether these generalized GK curves $\CC_n$ are covered by the Hermitian curve for $n\geq 5$. In this paper, we give a partial answer to this problem by showing that $\CC_n$ is not a Galois subcover of the Hermitian curve over the same finite field for any odd $n\geq 3$ and $q\geq 3$. More precisely, we prove the following.

\begin{theorem}\label{mainthm}
The generalized GK curve, defined by \eqref{GGKeqn} over $\mathbb{F}_{q^{2n}}$, is not a Galois subcover of the Hermitian curve over $\mathbb{F}_{q^{2n}}$ for any $q\geq 3$ and odd $n\geq 3$.
\end{theorem}

For the proof we make use of the Artin character for the Hermitian function field. This appears to be the first time that the Artin
character is used to study subcovers of the Hermitian function field and it allows us to avoid dealing with explicit equations and
group structures for Galois subcovers. The proof depends on Proposition \ref{proplb} which may be of independent interest and which gives a new lower bound
for the degree of a Galois covering $\phi: \CX\longrightarrow\CC$ of a maximal curve $\CC$ by a Hermitian curve $\CX$.

For $q=2$, the situation is different. For the case $n=3$, Giulietti and Korchm{\'a}ros show that the GK curve over $\mathbb{F}_{64}$ is covered by the Hermitian curve over the same field \cite{GiKo09}. We prove that if the generalized GK curve $\CC_n$ over $\mathbb{F}_{2^{2n}}$ is Galois covered by the Hermitian curve over $\mathbb{F}_{2^{2n}}$, then there is exactly one possibility for the degree and ramification structure.

\begin{theorem}\label{mainthmq=2}
Let $n\geq 5$ be an odd integer. If the generalized GK curve, defined by \eqref{GGKeqn} with $q=2$ over $\mathbb{F}_{2^{2n}}$, is Galois covered by the Hermitian curve over the same finite field, then the degree of the covering is $d=(2^n+1)/3$, and the covering is unramified.
\end{theorem}

Next, we consider the curve $\CX_n$ defined over $\mathbb{F}_{q^{2n}}$ by \eqref{Xneqn}
which is the second equation in the definition of the generalized GK curve. For $q=2$ it is shown in \cite{ABQ09} that $\CX_n$ is covered by the Hermitian curve, but it is not known whether this curve is a subcover of the Hermitian curve for $q\geq 3$. For $q=3$, Garcia and Stichtenoth \cite{GaSt06} showed that the curve is not a
Galois subcover of the Hermitian curve. Unlike the case for the generalized GK curve, our new lower bound does not eliminate all the possible degrees of a Galois covering. Nevertheless, we are able to shorten the interval of possible degrees that is obtained by the traditional method (which will be outlined in Section \ref{secsubH}).

\begin{theorem}\label{thmXn}
Let $n\geq 3$ be an odd integer. If the curve $\CX_n$ defined by \eqref{Xneqn} with $q>2$ is Galois covered by the Hermitian curve over the same finite field, then the degree of the covering $d$ satisfies
\begin{equation*}
\frac{(q+1)(q^n+1)}{q^2+1} \leq d \leq q^{n-1}+q^{n-2}+\ldots+q^2+q+2.
\end{equation*}
\end{theorem}

We remark that for $n\geq 5$, we do not know whether the curve $\CC_n$ is non-Galois covered by the Hermitian curve or not.

\section{Subcovers of the Hermitian curve} \label{secsubH}

Let $\CH_n$ be the Hermitian curve of degree $q^n+1$ over $\mathbb{F}_{q^{2n}}$. It has genus $g(\CH_n) = \frac{1}{2}q^n(q^n-1)$ and number of $\mathbb{F}_{q^{2n}}$-rational points $N(\CH_n) = q^{3n}+1.$ Let $\CY_n$ be a subcover of the Hermitian curve with morphism $\phi : \CH_n \longrightarrow \CY_n$ of degree $d$. From the splitting of points
we obtain a lower bound for $d$, and from the Hurwitz genus formula (see \cite[Theorem 3.4.13]{Sti09}) an upper bound for $d$,
\begin{equation} \label{eq:bounds}
\frac{N(\CH_n)}{N(\CY_n)} \leq d \leq \frac{2g(\CH_n)-2}{2g(\CY_n)-2}.
\end{equation}
Subcovers of the Hermitian curve are again maximal, and thus
\[
N(\CY_n) = q^{2n}+1+2g(\CY_n)q^n = (q^n+1)^2 + (2g(\CY_n)-2)q^n.
\]
The Hermitian curve has $2g(\CH_n)-2 = (q^n-2)(q^n+1)$. For $(A-1)(q^n+1) \leq 2g(\CY_n)-2 < A(q^n+1)$,
the bounds \eqref{eq:bounds} yield
\begin{equation*}
\frac{q^n}{A+1} \leq d \leq \frac{q^n-2}{A-1}.
\end{equation*}
The lower bound holds with equality for a covering $\phi : \CH_n \longrightarrow \CY_n$ of degree $d$ with
\begin{equation} \label{sharp}
2g(\CY_n)-2 = (q^n/d-1)(q^n+1)-(q^n/d+1)  \qquad \text{for $d|q^n$.}
\end{equation}
Such a covering exists for every divisor $d$ of $q^n$ (\cite[Section 3]{GSX00}).
For other cases we will use the following refinement of the lower bound.

\begin{lemma}\label{lemcov}
Let $\CY_n$ be a maximal curve with $2g(\CY_n)-2 = A(q^n+1) - B$, for integers $A$ and $B$ with $1 \leq B \leq q^n+1$. For $k(A+1) < B$, and for
$B \neq A+2$,
\[
\frac{q^n+k}{A+1} \leq d.
\]
In particular, $d(A+1) \geq q^n+1$ for $B > A+2.$
\end{lemma}

\begin{proof}
For the relevant case $B=k(A+1)+1$, the inequality $N(\CH_n)(A+1) > N(\CX_n)(q^n-1+k)$ reduces to
$\left( k q^n-1 \right)  \left( k-2 \right)+A+A \left( k-1 \right) ^{2}q^n > 0$, which holds for $k \geq 2$ and for $k=0$ .
For $k=1$, we need to verify only the case $B=A+3$. For $B=A+3$, the inequality
$N(\CH_n)(A+1) > N(\CX_n)(q^n)$ reduces to $q^{2n}-q^n+A+1 > 0$.
\end{proof}

Let $k$ be maximal with $k(A+1) < B.$ For the degree of the ramification divisor $R$ we write
\begin{equation} \label{eq:degRAB}
\deg R = (2g(\CH_n)-2)- d (2g(\CY_n)-2) = R_0 (q^n+1) + R_1,
\end{equation}
where $R_0 = (q^n-2-dA+k)$ and $R_1 = dB-k(q^n+1).$
For Galois subcovers,
we write the degree of the ramification divisor in a different way in the next two sections.
In Proposition \ref{proplb} we show that the combined descriptions exclude the possibility
$R_1 < q^n+1$, which yields a lower bound $d \geq (k+1)(q^n+1)/B$ that in many cases improves
the standard lower bound \eqref{eq:bounds}.

\section{The degree of the ramification divisor of a Galois covering} \label{secprelim} 

Now we suppose that the covering $\phi:\CH_n\longrightarrow\CY_n$ is Galois with Galois group $G$, with $\abs{G}=d$. Then $G$ can be realized as a subgroup of $\text{Aut}(\CH_n)=PGU(3,q^n)$ (see \cite{Leo96,Sti73}), and $\CY_n$ is the quotient curve of $\CH_n$ by $G$. To understand the ramification in a Galois covering, we will use the Hilbert different formula (see the proof in \cite[Theorem 3.8.7]{Sti09}), which we state here for the sake of completeness. Let $\CX\longrightarrow\CX'$ be a Galois covering of curves with Galois group $G$, and let $P$ and $P'$ be points on $\CX$ and $\CX'$ respectively (which need not lie in the field of definition of the covering) so that $P$ maps to $P'$ under the covering. Then the different exponent $d(P|P')$ is
\begin{equation}\label{eqndiff}
d(P|P')=\sum_{\substack{1\neq\sigma\in G \\ \sigma(P)=P}} i_{P}(\sigma),
\end{equation}
where $i_{P}(\sigma)=v_{P}(\sigma(t)-t)$ with $t$ a local uniformizer at $P$. Note that if the ramification of $P$ over $P'$ is tame, then $i_{P}(\sigma)=1$ for any $\sigma$ that fixes $P$, and in that case $d(P|P')$ is the number of elements $\sigma\neq 1$ in $G$ that fix $P$. For any $\sigma\in PGU(3,q^n)$, define
\begin{equation}\label{defisigma}
i(\sigma):=\sum_{P\in\CX} i_P(\sigma)\deg P,
\end{equation}
with the convention that $i_P(\sigma)=0$ if $\sigma(P)\neq P$. Combining \eqref{eqndiff} with the Hurwitz genus formula, we get the following proposition which we will rely on heavily.

\begin{prop}\label{propram}
Suppose $\CX\longrightarrow\CX'$ is a Galois covering of degree $d$ with Galois group $G$, then
\begin{equation*}
2g(\CX)-2=d(2g(\CX')-2)+\deg R,
\end{equation*}
where $R$ is the ramification divisor, whose degree is given by
\begin{equation*}
\deg R = \sum_{1\neq\sigma\in G} i(\sigma).
\end{equation*}
\end{prop}

\begin{remark}
We want to point out the relation between $i(\sigma)$ and the Artin representation (see \cite[Chapter 19]{Ser77}, \cite[Chapter VI]{Ser79}).
The character of the Artin representation satisfies $a(\sigma) = -i(\sigma)$, for $\sigma \neq 1$, and
$\sum_\sigma a(\sigma) = 0$. For the subcover of the Hermitian function field with group $PGU(3,q^n)$,
the Artin representation is the unique irreducible representation of minimal degree $2g(\CH_n)=q^n(q^n-1)$ (see \cite[Lemma 4.1]{LaSe74}).
\end{remark}

\section{Artin character of the Hermitian function field} \label{secartin}

To apply Proposition \ref{propram} with $\CX=\CH_n$ and $\CX'=\CY_n$, we need to understand the values of $i(\sigma)$ defined by \eqref{defisigma} for $\sigma\in PGU(3,q^n)$. The action of $PGU(3,q^n)$ on the Hermitian curve $\CH_n$ is well-known \cite{Sti73}. An element in $PGU(3,q^n)$ either fixes no points on $\CH_n$, or it fixes a point of degree one, or fixes a point of degree three. If $\sigma$ fixes no points on $\CH_n$, then $i(\sigma)=0$. If it fixes a point of degree three, then it fixes only that point. Since any such $\sigma$ has order dividing $q^{2n}-q^n+1$, which is relatively prime to $q$, the ramification is tame. Hence $i(\sigma)=3$. The case when $\sigma$ fixes a point of degree one has several subcases. Since the action of $PGU(3,q^n)$ on the points of degree one on $\CH_n$ is transitive (see for example \cite{HuPi73}), and $i(\sigma)$ is unchanged under conjugation
(see for example \cite[Chapter IV]{Ser79}), we may assume that the degree one point fixed is the point at infinity $P_{\infty}$ when $\CH_n$ is given by the equation $x^{q^n}+x=y^{q^n+1}$. Let $H$ be the subgroup of $PGU(3,q^n)$ fixing $P_{\infty}$. One can show that $H$ is of order $q^{3n}(q^{2n}-1)$, and any $\sigma\in H$ is of the form
\begin{align}\label{eqnsigma}
\sigma(x)&=a^{q^n+1}x+ab^{q^n}y+c, & \sigma(y)&=ay+b,
\end{align}
with $a\in \mathbb{F}_{q^{2n}}\backslash\{0\}$, $b\in\mathbb{F}_{q^{2n}}$, $c^{q^n}+c=b^{q^n+1}$. Following the notations in \cite{GSX00}, we denote by $\sigma=[a,b,c]$ the automorphism $\sigma\in H$ given by \eqref{eqnsigma}. There are $2$ cases.

\begin{lemma} 
Let $P_{\infty}$ be the point at infinity when $\CH_n$ is given by the equation $x^{q^n}+x=y^{q^n+1}$, and let $H$ be the subgroup of $PGU(3,q^n)$ fixing $P_{\infty}$. Let $\sigma=[a,b,c]\in H$ with $\sigma\neq 1$. For $a\neq 1$, we have
\begin{equation*}
i(\sigma)=\begin{cases}
1 &, \text{~if~} p \text{~divides~} \ord(\sigma), \\
q^n+1 &, \text{~if~} \ord(\sigma) \text{~divides~} q^n+1, \\
2 &, \text{~otherwise}.
\end{cases}
\end{equation*}
For $a=1$, we have
\begin{equation*}
i(\sigma)=\begin{cases}
2 &, \text{~if~} a=1,b\neq 0, \\
q^n+2 &, \text{~if~} a=1,b=0,c\neq 0.
\end{cases}
\end{equation*}

\end{lemma}
\begin{proof}
(Case $a \neq 1$)~The Sylow $p$-subgroup of $H$ is the set consisting of $[a,b,c]$ with $a=1$. Therefore, if $\sigma=[a,b,c]$ with $a\neq 1$, then $\sigma$ is not in the higher ramification group of $P_{\infty}$, so it will not fix $P_{\infty}$ to a high order. Since all other places is at most tamely ramified, we have
\[i_{P}(\sigma)=v_{P}(\sigma(t)-t)=
\begin{cases}
0 &, \sigma(P) \neq P, \\
1 &, \sigma(P)=P.
\end{cases}\]
Therefore, in this case we have
\[i(\sigma)=\#\{P\in\CH_n | \deg(P)=1 \text{~and~} \sigma(P)=P\}.\]
If $\ord(\sigma)$ is a multiple of $p$, then $\sigma$ cannot fix any degree one places other than $P_{\infty}$ since those places are tame. Thus $i(\sigma)=1$. Suppose now $\ord(\sigma)$ divides $q^n+1$, then one can show that $\sigma=[a,b,c]$ is conjugate in $H$ to $\sigma^{\ast}=[a,0,0]$ (see \cite[Lemma 4.1]{GSX00}). By \eqref{eqnsigma}, $\sigma^{\ast}$ satisfies $\sigma^{\ast}(x)=x$ and $\sigma^{\ast}(y)=ay$. It is then easy to see that in the affine part of $\CH_n$, $\sigma^{\ast}$ fixes exactly the (affine) line $\{y=0\}$. Hence, $i(\sigma^{\ast})=q^n+1$ as $\#(\CH_n \cap \{y=0\})=q^n$ and $\sigma^{\ast}$ also fixes $P_{\infty}$. Since $i(\sigma)$ is preserved under conjugation, we have $i(\sigma)=i(\sigma^{\ast})=q^n+1$. Finally, if the order of $\sigma$ does not divide $q^n+1$, then again $\sigma=[a,b,c]$ is conjugate in $H$ to $\sigma^{\ast}=[a,0,0]$. This time we have $\sigma^{\ast}(x)=a^{q^n+1}x$ and $\sigma^{\ast}(y)=ay$. So in the affine part of $\CH_n$, $\sigma^{\ast}$ fixes exactly the origin. Thus $i(\sigma)=i(\sigma^{\ast})=2$.

(Case $a = 1$)~If $\sigma=[a,b,c]$ with $a=1$, then $\sigma$ is in the higher ramification group of $P_{\infty}$, and this is the only point that $\sigma$ can fix. In this case, we compute $i(\sigma)$ directly from the definition. First, $i(\sigma)=i_{P_{\infty}}(\sigma)=v_{P_{\infty}}(\sigma(t)-t)$, where $t$ is a local uniformizer at $P_{\infty}$. We choose $t=y/x$ to be the local uniformizer. Then
\begin{align*}
i(\sigma)&=v_{P_{\infty}}(\frac{y+b}{x+b^{q^n}y+c}-\frac{y}{x}) \\
&= v_{P_{\infty}}((y+b)x-y(x+b^{q^n}y+c))-v_{P_{\infty}}(x)-v_{P_{\infty}}(x+b^{q^n}y+c) \\
&= v_{P_{\infty}}(-b^{q^n}y^2+bx-cy)+2(q^n+1) \\
&= \begin{cases}
2 &, \text{~if~} b\neq 0, \\
q^n+2 &, \text{~if~} b=0,c\neq 0.
\end{cases}
\end{align*}
\end{proof}

The $i(\sigma)$ among various kinds of elements are shown in Figure \ref{fig1}. Each number in a box corresponds to a subgroup of that order, and each number on an edge is the $i(\sigma)$ for elements that lie in the upper group but not the lower one.

\begin{figure}[h]
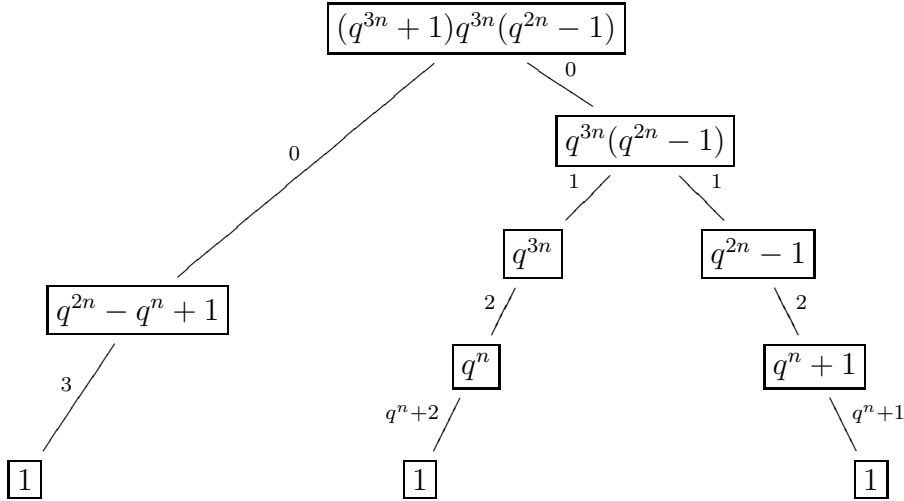

\begin{equation*}
\xygraph{
!{<0cm,0cm>;<1.5cm,0cm>:<0cm,1.5cm>::}
!{(4,0) }*+{\framebox{$(q^{3n}+1)q^{3n}(q^{2n}-1)$}}="a"
!{(5.5,-1) }*+{\framebox{$q^{3n}(q^{2n}-1)$}}="b"
!{(6.5,-2) }*+{\framebox{$q^{2n}-1$}}="c"
!{(7,-3)}*+{\framebox{$q^n+1$}}="d"
!{(7.5,-4)}*+{\framebox{$1$}}="e"
!{(4.5,-2)}*+{\framebox{$q^{3n}$}}="f"
!{(4,-3)}*+{\framebox{$q^n$}}="g"
!{(3.5,-4)}*+{\framebox{$1$}}="h"
!{(1,-2.5)}*+{\framebox{$q^{2n}-q^n+1$}}="i"
!{(0,-4)}*+{\framebox{$1$}}="j"
"a"-^{0}"b" "b"-^{1}"c" "c"-^{2}"d" "d"-^{q^n+1}"e" "b"-_{1}"f" "f"-_{2}"g" "g"-_{q^n+2}"h" "a"-_{0}"i" "i"-_{3}"j"
}
\end{equation*}
\caption{$i(\sigma)$ among various elements in $PGU(3,q^n)$ \label{fig1}}
\end{figure}

The following proposition follows immediately from the above lemma. The significance of the proposition is that either $i(\sigma)$ is very small, or $i(\sigma)$ is very large, and nothing in the middle can happen.
\begin{prop}\label{propis}
If $\sigma\in PGU(3,q^n)$, then $i(\sigma)=0,1,2,3,q^n+1$ or $q^n+2$.
\end{prop}

Now we have the contribution of each element in $PGU(3,q^n)$ to the ramification divisor, and we are ready to finish the proof of the theorems stated in the introduction.

\section{Galois subcovers of the Hermitian curve}\label{secpf}

In Proposition \ref{propram}, we write the degree of the ramification divisor $R$ as the sum of $i(\sigma)$ for the $d-1$ nontrivial $\sigma$ in the Galois group $G$, and in Proposition \ref{propis} we found the possible values for each $i(\sigma)$. In particular, the nontrivial elements divide into two groups according to $i(\sigma) = 0,1,2,3$ or $i(\sigma) = q^n+1, q^n+2$. Write $d=1+u+v$ with
\begin{equation*}
u = \# \{ \sigma \neq 1 : i(\sigma) = 0,1,2,3 \} \text{~~and~~} v = \# \{ \sigma \neq 1 : i(\sigma) = q^n+1, q^n+2 \}.
\end{equation*}
Then
\begin{equation} \label{eq:degRuv}
v(q^n+1) \leq \deg R \leq v(q^n+1)+3u+v.
\end{equation}
We compare this with the description of $\deg R$ in Section \ref{secsubH}.
For a subcover $\CY_n$ of the Hermitian curve $\CH_n$, not necessarilly
Galois, let $2g(\CY_n)-2 = A(q^n+1) - B$, with $1 \leq B \leq q^n+1$,
and let $k$ be maximal with $k(A+1) < B.$ As in \eqref{eq:degRAB}, let
\begin{equation} \label{eq:degRAB2}
\deg R = (2g(\CH_n)-2)- d (2g(\CY_n)-2) = R_0 (q^n+1) + R_1,
\end{equation}
where $R_0 = (q^n-2-dA+k)$ and $R_1 = dB-k(q^n+1).$ Clearly,
\begin{equation} \label{eq:R10}
k(R_0-d)+(R_1-d) \geq k(k-3). 
\end{equation}
We will now prove a new lower bound for $d$.

\begin{prop}\label{proplb}
Let $\CY_n$ be a maximal curve with $2g(\CY_n)-2 = A(q^n+1) - B$, for integers $A$ and $B$ with $1 \leq B \leq q^n+1$.
For $B > A+2$ and for $k(A+1) < B$, if $\phi : \CH_n \longrightarrow \CY_n$ is a Galois covering of degree $d$
then $dB \geq (k+1)(q^n+1).$
\end{prop}

\begin{proof}
Assume to the contrary that $dB < (k+1)(q^n+1)$. Since $B \geq 2k+1$, $3d < 2(q^n+1),$ and thus $3u+v < 2(q^n+1).$
%
With Lemma \ref{lemcov}, $dB > dk(A+1) \geq k(q^n+1)$. Together with the assumption,
\[
k(q^n+1) < dB < (k+1)(q^n+1).
\]
For $\deg R = R_0(q^n+1)+R_1$ in \eqref{eq:degRAB2}, it follows that $R_0$ corresponds to the quotient and $R_1$ to the remainder after
divisor by $q^n+1$. Now we compare with \eqref{eq:degRuv}. Using $R_0 \leq v+1$ and $R_1 \leq 3u+v$,
\[
k(R_0-d)+(R_1-d) \leq k(-u)+2u-1,
\]
which, for $k \geq 3$, contradicts \eqref{eq:R10}. It remains to prove the case ($k=2$) and the case ($k=1, B > A+2$).
Observe that for ($k=1, B > A+2$), \eqref{eq:R10} can be replaced with
\begin{equation} \label{eq:k1}
(R_0-d)+(R_1-d) \geq d-2.
\end{equation}
If $3u+v < q^n+1$ then $R_0 = v$ and $R_1 \leq 3u+v$. For ($k=2$), $2R_0 + R_1 \leq 3u+3v = 3d-3$ contradicts \eqref{eq:R10}.
For ($k=1, B > A+2$), $R_0+R_1 \leq 3u+2v = 2d+u-2$ contradicts \eqref{eq:k1}. If $3u+v \geq q^n+1$ then $R_0 \leq v+1$ and
$R_1 \leq 3u+v-1$. For ($k=2$), $2R_0 + R_1 \leq 3u+3v+1 = 3d-2$. In combination with \eqref{eq:R10} equality holds and
$R_0 = v+1$ and $R_1 = 3u+v-1$. The latter implies $3u+v = q^n+1$, and $R_0 = v+1$ would then imply $R_1 = 0$, a contradiction.
For ($k=1, B > A+2$), $R_0+R_1 \leq 3u+2v = 2d+u-2$ contradicts \eqref{eq:k1}.
\end{proof}

We will apply the above proposition to the generalized GK curve $\CC_n$ and the plane curve $\CX_n$, that are both maximal curves over $\mathbb{F}_{q^{2n}}$. For their genera we have
\begin{align}
2g(\CH_n)-2 &~=~ (q^n-2)(q^n+1), \nonumber \\
2g(\CC_n)-2 &~=~ (q^2-1)(q^n+1)-(q^3+1), \label{eq:gCn} \\
2g(\CX_n)-2 &~=~ (q-1)(q^n+1)-(q^2+1). \label{eq:gXn}
\end{align}

We first consider the generalized GK curve $\CC_n$. Suppose now that $\phi:\CH_n\longrightarrow\CC_n$ is a Galois covering of degree $d$. From \eqref{eq:gCn}, we have $A=q^2-1$, $B=q^3+1$ and $k=q$.
Proposition \ref{proplb} gives the lower bound for $d$ as
\begin{equation*}
d \geq \frac{(k+1)(q^n+1)}{B} = \frac{q^n+1}{q^2-q+1}.
\end{equation*}
From \eqref{eq:bounds} we have the upper bound, for $n \geq 3$,
\begin{equation*}
d \leq \frac{2g(\CH_n)-2}{2g(\CC_n)-2} \leq \frac{q^n-2}{q^2-2}.
\end{equation*}
For $q \geq 3$ and $n \geq 3$ the lower bound exceeds the upper bound and no solutions for $d$ exist. Hence the GK curve cannot be a Galois subcover of the Hermitian curve. This is Theorem \ref{mainthm}.

For $q=2$, $2g(\CH_n)-2 = (2^n+1)/3 \cdot (2g(\CX_n)-2)$ and the inequalities admit the unique solution $d = (2^n+1)/3$.
This gives the degree in Theorem \ref{mainthmq=2}. Moreover, Proposition \ref{propram} reveals that such a covering, if it exists, has to be unramified. This proves Theorem \ref{mainthmq=2}.

\begin{remark}
In the proof we did not use the fact that we are dealing with the generalized GK-curve $\CC_n$. What we use is only the genus of $\CC_n$ given by \eqref{eq:gCn}. Thus we actually prove that there are no curves with genus $\frac{1}{2}(q-1)(q^{n+1}+q^n-q^2)$ being a Galois subcover of the Hermitian curve $\CH_n$ when $q\geq 3$ and odd $n \geq 3$.
\end{remark}

We now turn our attention to $\CX_n$. Suppose that $\phi:\CH_n\longrightarrow\CX_n$ is a Galois covering of degree $d$. From \eqref{eq:gXn}, we have $A=q-1$, $B=q^2+1$ and $k=q$. Proposition \ref{proplb} and \eqref{eq:bounds} gives the lower and upper bounds for $d$ as
\begin{equation*}
\frac{(q+1)(q^n+1)}{q^2+1} \leq d \leq q^{n-1}+q^{n-2}+\ldots+q^2+q+2.
\end{equation*}
This proves Theorem \ref{thmXn}.

\begin{remark}
In most cases we expect that one need to corporate our ideas with other methods to completely remove the possibility of a curve being a Galois subcover of the Hermitian curve. For a concrete example, consider $n=3$. Then $\CX_3$ is given by the equation $y^{q^2}-y=z^{q^2-q+1}$ and has genus $\frac{1}{2}(q-1)(q^3-q)$. Our bounds for the degree $d$ gives $q^2+q \leq d \leq q^2+q+2$. We can eliminate the cases $d=q^2+q+1$ and $d=q^2+q+2$ by some elementary arguments. However, the case $d=q^2+q$ is more subtle. It turns out that there is a curve
\[ \CY: y^{q^2}-y^q+y=x^{q^2-q+1} \]
which has the same genus as $\CX_3$, and is a Galois subcover of the Hermitian curve of degree $q^2+q$, with
\[\begin{array}{c|*6{@{\hspace{6mm}}c}}
i(\sigma) & 0 & 1 & 2 & 3 & q^3+1 & q^3+2 \\
\hline
\# \sigma & 1 & q^2-q & 0 & 0 & q & q-1 \\
\end{array}\]
It can be proved by the same idea as the case $q=3$ in \cite{GaSt06} that $\CY$ is not isomorphic to $\CX_3$, but that does not settle the case $d=q^2+q$ completely.
\end{remark}

\subsection*{Acknowledgments}
We would like to express our gratitude to Professor Mike Zieve and Professor Rachel Pries for pointing out a mistake in an earlier version of the preprint.


\begin{thebibliography}{10}

\bibitem{ABQ09}
Miriam Abd{\'o}n, Juscelino Bezerra, and Luciane Quoos, \emph{Further examples
  of maximal curves}, J. Pure Appl. Algebra \textbf{213} (2009), no.~6,
  1192--1196.

\bibitem{AbGa04}
Miriam Abd{\'o}n and Arnaldo Garcia, \emph{On a characterization of certain
  maximal curves}, Finite Fields Appl. \textbf{10} (2004), no.~2, 133--158.

\bibitem{AbTo99}
Miriam Abd{\'o}n and Fernando Torres, \emph{On maximal curves in characteristic
  two}, Manuscripta Math. \textbf{99} (1999), no.~1, 39--53.

\bibitem{CKT99}
Antonio Cossidente, Gabor Korchm{\'a}ros, and Fernando Torres, \emph{On curves
  covered by the {H}ermitian curve}, J. Algebra \textbf{216} (1999), no.~1,
  56--76.

\bibitem{CKT00}
\bysame, \emph{Curves of large genus covered by the {H}ermitian curve}, Comm.
  Algebra \textbf{28} (2000), no.~10, 4707--4728.

\bibitem{Duu10}
Iwan~M. Duursma, \emph{Two-point coordinate rings for {GK}-curves}, IEEE Trans. Inform. Theory \textbf{57} (2011), no.~2, 593--600.

\bibitem{FGT97}
Rainer Fuhrmann, Arnaldo Garcia, and Fernando Torres, \emph{On maximal curves},
  J. Number Theory \textbf{67} (1997), no.~1, 29--51.

\bibitem{FuTo96}
Rainer Fuhrmann and Fernando Torres, \emph{The genus of curves over finite
  fields with many rational points}, Manuscripta Math. \textbf{89} (1996),
  no.~1, 103--106.

\bibitem{Gar01}
Arnaldo Garcia, \emph{Curves over finite fields attaining the {H}asse-{W}eil
  upper bound}, European {C}ongress of {M}athematics, {V}ol. {II} ({B}arcelona,
  2000), Progr. Math., vol. 202, Birkh\"auser, Basel, 2001, pp.~199--205.

\bibitem{GGS10}
Arnaldo Garcia, Cem G{\"u}neri, and Henning Stichtenoth, \emph{A generalization
  of the {G}iulietti-{K}orchm\'aros maximal curve}, Adv. Geom. \textbf{10}
  (2010), no.~3, 427--434.

\bibitem{GaSt06}
Arnaldo Garcia and Henning Stichtenoth, \emph{A maximal curve which is not a
  {G}alois subcover of the {H}ermitian curve}, Bull. Braz. Math. Soc. (N.S.)
  \textbf{37} (2006), no.~1, 139--152.

\bibitem{GSX00}
Arnaldo Garcia, Henning Stichtenoth, and Chao-Ping Xing, \emph{On subfields of
  the {H}ermitian function field}, Compositio Math. \textbf{120} (2000), no.~2,
  137--170.

\bibitem{GiKo09}
Massimo Giulietti and G{\'a}bor Korchm{\'a}ros, \emph{A new family of maximal
  curves over a finite field}, Math. Ann. \textbf{343} (2009), no.~1, 229--245.

\bibitem{HuPi73}
Daniel~R. Hughes and Fred~C. Piper, \emph{Projective planes}, Springer-Verlag,
  New York, 1973, Graduate Texts in Mathematics, Vol. 6.

\bibitem{Iha81}
Yasutaka Ihara, \emph{Some remarks on the number of rational points of
  algebraic curves over finite fields}, J. Fac. Sci. Univ. Tokyo Sect. IA Math.
  \textbf{28} (1981), no.~3, 721--724.

\bibitem{KoTo01}
G{\'a}bor Korchm{\'a}ros and Fernando Torres, \emph{Embedding of a maximal
  curve in a {H}ermitian variety}, Compositio Math. \textbf{128} (2001), no.~1,
  95--113.

\bibitem{KoTo02}
\bysame, \emph{On the genus of a maximal curve}, Math. Ann. \textbf{323}
  (2002), no.~3, 589--608.

\bibitem{Lac87}
Gilles Lachaud, \emph{Sommes d'{E}isenstein et nombre de points de certaines
  courbes alg\'ebriques sur les corps finis}, C. R. Acad. Sci. Paris S\'er. I
  Math. \textbf{305} (1987), no.~16, 729--732.

\bibitem{LaSe74}
Vicente Landazuri and Gary~M. Seitz, \emph{On the minimal degrees of projective
  representations of the finite {C}hevalley groups}, J. Algebra \textbf{32}
  (1974), 418--443.

\bibitem{Leo96}
Heinrich-Wolfgang Leopoldt, \emph{\"{U}ber die {A}utomorphismengruppe des
  {F}ermatk\"orpers}, J. Number Theory \textbf{56} (1996), no.~2, 256--282.

\bibitem{RuSt94}
Hans-Georg R{\"u}ck and Henning Stichtenoth, \emph{A characterization of
  {H}ermitian function fields over finite fields}, J. Reine Angew. Math.
  \textbf{457} (1994), 185--188.

\bibitem{Ser77}
Jean-Pierre Serre, \emph{Linear representations of finite groups},
  Springer-Verlag, New York, 1977, Translated from the second French edition by
  Leonard L. Scott, Graduate Texts in Mathematics, Vol. 42.

\bibitem{Ser79}
\bysame, \emph{Local fields}, Graduate Texts in Mathematics, vol.~67,
  Springer-Verlag, New York, 1979, Translated from the French by Marvin Jay
  Greenberg.

\bibitem{Sti73}
Henning Stichtenoth, \emph{\"{U}ber die {A}utomorphismengruppe eines
  algebraischen {F}unktionenk\"orpers von {P}rimzahlcharakteristik. {I, II}},
  Arch. Math. (Basel) \textbf{24} (1973), 527--544, 615--631.

\bibitem{Sti09}
\bysame, \emph{Algebraic function fields and codes}, second ed., Graduate Texts
  in Mathematics, vol. 254, Springer-Verlag, Berlin, 2009.

\bibitem{StXi95}
Henning Stichtenoth and Chao~Ping Xing, \emph{The genus of maximal function
  fields over finite fields}, Manuscripta Math. \textbf{86} (1995), no.~2,
  217--224.

\end{thebibliography}

\end{document}